\documentclass[preprint]{article}
\usepackage{amsfonts}
\usepackage{amssymb}
\usepackage{amsthm}
\usepackage{amsmath}
\usepackage{centernot}
\usepackage{tikz}
\usetikzlibrary{matrix}
\newtheorem{thm}{Theorem}
\newtheorem{lem}{Lemma}
\newtheorem{defin}{Definition}

\usepackage{hyperref}
\usepackage{cite}

\title{One-sided Duo rings are McCoy}
\author{Michael Cary\footnote{Department of Economics and Department of Mathematics, West Virginia University, Morgantown, WV 26506, email: macary@mix.wvu.edu}}

\begin{document}
\maketitle
\begin{abstract}
In this paper we prove that one-sided Duo rings are (two-sided) McCoy.  By doing so, we are then able to explicitly describe some of these ring element annihilators of polynomials in McCoy rings.  We conclude the paper by showing the place of these results in the literature by way of an extension of a convenient diagram from \cite{Camillo1}.
\end{abstract}

Keywords: McCoy ring; Duo ring; annihilation

MSC[2010]: 16-U99; 16-W25; 16-K99

\section{Introduction}
The purpose of this paper is to extend results on the relationships between McCoy rings and other ring-theoretic structures, particularly by proving that one-sided Duo rings are (two-sided) McCoy, as well as to describe some properties of ring element annihilators of polynomials in the polynomial ring over the one-sided Duo ring.

McCoy rings are a well-studied ring-theoretic structure, with results dating back to the seminal work by N.H. McCoy which shows that polynomials which annihilate one another over a commutative ring each admit an annihilator in the base ring in, e.g. \cite{McCoy1} and \cite{McCoy2}.  Since then, many authors have studied this and other highly related properties, for example \cite{Camillo1,Camillo2,Hong,Hirano2,Hirano1,KimB,KimN,Marks1,Nielsen} and \cite{Shin}, resulting in a well articulated analysis of the annihilation of polynomials and of zero divisors in various specific (and sometimes more general) contexts of ring theory and ring-theoretic structures.  Duo rings have a similarly thorough and longevous history, for example \cite{Habeb}, \cite{Marks2}, and \cite{Ziembowski}.  We extend some of these results in this paper from a structural standpoint by studying relationships between Duo rings and McCoy rings.  Other ring-theoretic structures highly relevant to this work, structures such as symmetric rings, semi-commutative rings, and 2-primal rings, have a rich history on their relationships between McCoy rings, e.g. \cite{Marks1,Marks2,Nielsen} and \cite{Shin}.

We first show in Section \ref{sec2} that one-sided Duo rings are necessarily McCoy.  In \cite{Camillo1} it was proven that Right Duo rings are necessarily Right McCoy, but the opposite side relationship remains unproven, hence we begin this paper with a proof that Left Duo rings are necessarily Right McCoy.  Before proving this, we revisit an important lemma from that paper which proves quite useful in proving that Left Duo rings indeed are Left McCoy.  This fact, that Left Duo rings are Left McCoy, is an incredibly important result in proving that Left Duo rings are Right McCoy, and consequentially that one-sided Duo rings are McCoy.

In Section \ref{sec3} we describe some of these ring element annihilators of polynomials in polynomial rings over one-sided Duo base rings.  Finally, in our conclusory section, we demonstrate where this work fits into the existing literature by extending a very nice diagram from \cite{Camillo1}.  We present the standing version of the diagram here to illustrate what is known of the relationships between McCoy rings and other ring theoretic structures.

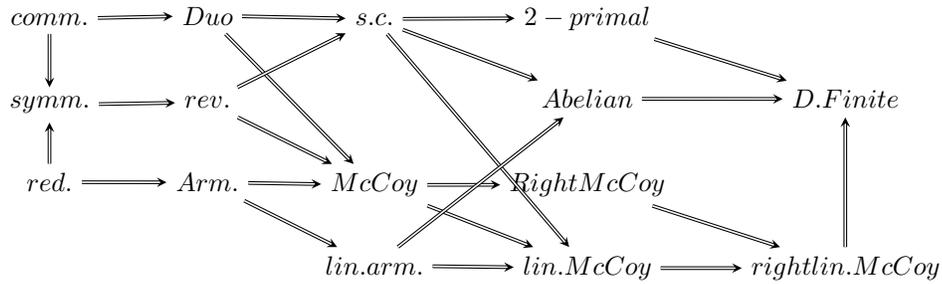
\begin{center}
\begin{tikzpicture}
\matrix (m) [matrix of math nodes, row sep=1.6em, column sep=2.6em, minimum width=2em]
{
comm. & Duo & s.c. & 2-primal\\
symm. & rev. & & Abelian & D. Finite\\
red. & Arm. & McCoy & Right McCoy\\
	& & lin. arm. & lin. McCoy & right lin. McCoy\\};
\path[-stealth]
(m-1-1) edge [double] (m-1-2)
(m-1-1) edge [double] (m-2-1)
(m-1-2) edge [double] (m-1-3)
(m-1-2) edge [double] (m-3-3)
(m-1-4) edge [double] (m-2-5)
(m-2-1) edge [double] (m-2-2)
(m-2-2) edge [double] (m-1-3)
(m-2-2) edge [double] (m-3-3)
(m-1-3) edge [double] (m-2-4)
(m-1-3) edge [double] (m-4-4)
(m-1-3) edge [double] (m-1-4)
(m-2-4) edge [double] (m-2-5)
(m-3-1) edge [double] (m-2-1)
(m-3-1) edge [double] (m-3-2)
(m-3-2) edge [double] (m-3-3)
(m-3-2) edge [double] (m-4-3)
(m-3-3) edge [double] (m-3-4)
(m-3-3) edge [double] (m-4-4)
(m-3-4) edge [double] (m-4-5)
(m-4-3) edge [double] (m-2-4)
(m-4-3) edge [double] (m-4-4)
(m-4-4) edge [double] (m-4-5)
(m-4-5) edge [double] (m-2-5);
\end{tikzpicture}
\end{center}

\section{One-sided Duo rings are necessarily McCoy}\label{sec2}
It is known from \cite{Camillo1} that Left (resp. Right) Duo rings are Left (Right) McCoy.  However, it is unknown whether or not Left Duo rings are Right McCoy.  In fact, this was a question posed by the authors in \cite{Camillo1}.  In this section we prove that Left Duo rings indeed are necessarily Right McCoy and thereby prove that one-sided Duo rings are McCoy.  For clarity and convenience, we first state some definitions.

\begin{defin}\label{defin1}
A ring, \emph{R}, is called Left Duo if every left ideal $\emph{I} \subset \emph{R}$ is a two-sided ideal.  This implies that for all $r\in\emph{R}$ and $x\in\emph{I}$ we have that $rx\in\emph{I}$ and $xr\in\emph{I}$.  Moreover, for some $r\in\emph{R}$ and $x\in\emph{I}$  we have that $rx=x'r$ for some $x'\in\emph{I}$. 
\end{defin}
\begin{defin}\label{defin2}
A ring is simply said to be Duo if it is both Left and Right Duo.
\end{defin}
\begin{defin}\label{defin3}
A ring, \emph{R}, is said to be Right McCoy if for $f(x),g(x)\in\emph{R}[x]\setminus\{0\}$ such that $f(x)g(x)=0$ implies that there exists an $r\in\emph{R}$ such that $f(x)r=0$.
\end{defin}

In order to prove this implication, that Left Duo rings are necessarily Right McCoy, we first state two lemmas.  For the first lemma, we revisit a lemma by Nielsen from \cite{Camillo1} and \cite{Nielsen} which will serve as the basis for describing ring element right annihilators of the polynomial $f(x)$ in McCoy rings which are also one-sided Duo rings.  It gives an explicit description of two ring elements which serve as left annihilators of the polynomial $g(x)$ in semi-commutative rings (note that one-sided Duo does indeed imply semi-commutative).  This lemma can also be found in \cite{Camillo1} as Lemma $5.4$ and in \cite{Nielsen} as Lemma $1$.  The second lemma can be found in \cite{Camillo1} as Theorem $8.2$.  It shows that Right Duo rings are necessarily Right McCoy, a fact which, when complemented by a proof that Left Duo rings are Right McCoy, is necessary in proving that one-sided Duo rings are McCoy.    

\begin{lem}\label{lem1}
Let \emph{R} be a semi-commutative ring, let $m,n\in\mathbb{N}$, and let $f(x)=\sum\limits_{i=0}^m a_{i}x^i\in\emph{R}[x]$ and $g(x)=\sum\limits_{i=0}^n b_{i}x^i\in\emph{R}[x]\setminus\{0\}$.  If there exists $g(x)\in\emph{R}[x]$ with $f(x)g(x)=0$ then $a_{0}^{n+1}g(x)=0$ and $a_{m}^{n+1}g(x)=0$.
\end{lem}
\begin{proof}
First notice that we may let $f^{\star}(x)=x^{m}f(x^{-1})$ and $g^{\star}(x)=x^{n}g(x^{-1})$, in which case we have merely reversed the coefficients of the polynomials $f(x)$ and $g(x)$.  Then, since $f(x)g(x)=0$, we have that $f^{\star}(x)g^{\star}(x)=0$.  By proving that $a_{0}^{n+1}g(x)=0$, we also prove that $a_{m}^{n+1}g^{\star}(x)=0$ by symmetry, and since $a_{m}^{n+1}g^{\star}(x)$ and $a_{m}^{n+1}g(x)$ are equivalent statements, hence it suffices to prove that $a_{0}^{n+1}g(x)=0$.  

Obviously the case of $i=0$ yields $a_{0}b_{0}=0$, and so we assume as our inductive hypothesis that $a_{i}b_{0}^{i+1}=0$ for all $i<j$.  Consider the coefficient on the $x^{j}$ term, namely $\sum\limits^{j}_{i=0}a_{i}b_{j-i}=0$.  Left multiplying by $a_{0}^{j}$ yields
\begin{equation}\label{eq1}
\sum\limits^{j}_{i=0}a_{0}^{j}a_{i}b_{j-1}=0
\end{equation}
Isolating the case of $i=0$ and distinguishing this term from the rest of the summation then yields
\begin{equation}\label{eq2}
a_{0}^{j+1}b_{j}+\sum\limits^{j}_{i=1}a_{0}^{j}a_{i}b_{j-i}=0
\end{equation}
But by semi-commutativity and our induction hypothesis, $\sum\limits^{j}_{i=1}a_{0}^{j}a_{i}b_{j-i}=0$, hence we obtain
\begin{equation}\label{eq3}
a_{0}^{j+1}b_{j}=0
\end{equation}
and by induction the proof is complete.
\end{proof}

\begin{lem}\label{lem2}
Right Duo rings are necessarily Right McCoy.
\end{lem}
\begin{proof}
Let \emph{R} be a Right Duo ring and let $f(x),g(x)\in\emph{R}[x]\setminus\{0\}$ with $f(x)g(x)=0$ where $f(x)=\sum\limits^{m}_{i=0}a_{i}x^{i}$ and $g(x)=\sum\limits^{n}_{j=0}b_{j}x^{j}$.  Moreover, let $I_{g(x)}$ denote the left ideal generated by the coefficients of $g(x)$.  The proof is by induction on the degree of the polynomial $f(x)$ and asserts that the nonzero ring element which acts as a right annihilator of $f(x)$ is contained in $I_{g(x)}$.

As our basis for induction, when $deg(f(x))=0$ we may take a minimal $j$ such that $b_{j}$ is nonzero and observe that $f(x)b_{j}=0$, hence $b_{j}$ is a nonzero ring element which annihilates $f(x)$ on the right and is contained in $I_{g(x)}$.

Now suppose our inductive hypothesis, that for $deg(f(x))<n$ there is a nonzero element of $I_{g(x)}$ which annihilates $f(x)$ on the right.  We consider the case where $deg(f(x))=n$.

Case $1$:  Suppose $a_{0}g(x)=0$.  Then $a_{0}I_{g(x)}=0$, so we may let $f^{\star}(x)=(f(x)-a_{0})/x$ and retain $f^{\star}(x)g(x)=0$.  But now we have that $deg(f(x))=n-1<n$, so by our inductive hypothesis there exists a nonzero element $r\in I_{g(x)}$ which annihilates $f^{\star}(x)$ on the right.  Considering the remaining term, $a_{0}r$, we may simply right multiply by $b_{0}$ to obtain $a_{0}rb_{0}$.  By semi-commutativity, $a_{0}b_{0}=0 \implies a_{0}rb_{0}=0$, so in any event we may take $rb_{0}\in I_{g(x)}$ as our right annihilator.

Case $2$:  Suppose $a_{0}g(x)\neq 0$.  Let $j$ be minimal such that $a_{0}b_{j}\neq 0$.  By Lemma \ref{lem1} there exists an integer $k>0$ such that $a_{0}^{k}b_{j}\neq 0=a_{0}^{k+1}b_{j}$.  By \emph{R} being Right Duo, there necessarily exists some $r\in\emph{R}$ satisfying $a_{0}^{k}b_{j}=b_{j}r$.  Hence we may let $g^{\star}(x)=g(x)r$ and observe that $f(x)g^{\star}(x)=0$ is maintained.  Moreover, $(0)\neq I_{g^{\star}(x)}\subset I_{g(x)}$, so we may replace $g(x)$ with $g^{\star}(x)$ without loss of generality.  Via this construction, $a_{0}$ now annihilates the first $j$ coefficients of $g^{\star}(x)$, therefore, after a finite number of repetitions of this process $a_{0}$ annihilates $g^{\star}(x)$ entirely, and we may revert back to Case $1$ to complete the proof.
\end{proof}

The analogous statement, that Left Duo rings are necessarily Left McCoy, the variant which we shall use in proving the Left Duo rings are necessarily McCoy, was proven in \cite{Nielsen} via the implication that one-sided Duo rings are semi-commutative which in turn are necessarily Left McCoy.  Moreover, by Lemma \ref{lem2} we know that Right Duo rings are Right McCoy, hence Right Duo rings are McCoy.  Thus it only needs to be proven, but particularly needs to be proven, that Left Duo rings either are or are not Right McCoy, since Left Duo rings are obviously Left McCoy.  Note that does not suffice to claim that Left Duo implies semi-commutative here, for while semi-commutative rings are always Left McCoy, they are \emph{not always Right McCoy}.

We are nearly ready to prove that one-sided Duo rings are necessarily McCoy.  For the sake of the proof, we introduce one last result from \cite{Camillo1} on degree considerations of polynomials in McCoy rings.

\begin{defin}\label{defin4}
A ring, \emph{R}, is said to be $(m,n)$-Right McCoy if $deg(f(x))\leq m$, $deg(g(x))\leq n$, and $f(x)g(x)=0$ together imply that there exist a non-zero $r\in \emph{R}$ such that $f(x)r=0$.
\end{defin}

\begin{thm}\label{thm1}
Left Duo rings are necessarily Right McCoy.
\end{thm}
\begin{proof}
The proof is by induction on the degree considerations of \emph{R}.  Let \emph{R} be a Left Duo ring with $f(x),g(x)\in\emph{R}[x]\setminus\{0\}$ with $f(x)g(x)=0$.  Denote by $I_{g(x)}$ the left ideal generated by the coefficients of the polynomial $g(x)$.  Since \emph{R} is Left Duo, \emph{R} is semi-commutative and thus Linearly McCoy, and trivially also Right Linearly McCoy.  Hence we know that for \emph{R} being Let Duo implies that \emph{R} is $(1,1)$-Right McCoy, and our basis for induction is established.

Next, assume that if \emph{R} is Left Duo, it is also $(m,n)$-Right McCoy as our inductive hypothesis.  Two cases naturally arise, each with a pair of subcases.  First, let $f(x)g(x)=0$ with $deg(f(x))=m+1$ and $deg(g(x))=n$.  If it is the case that $a_{0}g(x)=0$, then we may simply let $f^{\star}(x)=(f(x)-a_{0})/x$ and observe that, by our inductive hypothesis, there exists an $r\in\emph{R}\setminus\{0\}$ such that $f^{\star}(x)r=0$, whence we obtain a non-zero ring element annihilator of $f(x)$ and that \emph{R} is $(m+1,n)$-Right McCoy.  Thus we may assume that $a_{0}g(x)\neq 0$.

In this case, choose an arbitrary index $j$ such that $a_{0}b_{j}\neq 0$.  Since it must be the case that $j>0$, we may then let $g^{\star}(x)=g(x)b_{0}$ and obtain $a_{0}b_{j}b_{0}$ which is equal to $b_{j}^{\star}a_{0}b_{0}$ with $0\neq b_{j}^{\star}\in I_{g(x)}$ by the Left Duo property on \emph{R} and by use of the left ideal $I_{g(x)}$.  We readily observe that for an arbitrary $j>0$ such that $a_{0}b_{j}\neq 0$, $a_{0}b_{j}b_{0}=0$, hence by $a_{0}b_{0}=0$ we have that $a_{0}g^{\star}(x)=0$ where $g^{\star}(x)=g(x)b_{0}$ as before.  Thus if \emph{R} is Left Duo and $(m,n)$-Right McCoy, it is necessarily $(m+1,n)$-Right McCoy.

Now, for the second case, let \emph{R} be Left Duo and $(m,n)$-McCoy as described in the first case.  This time we consider the degree of $g(x)$.  Let $deg(f(x))=m$ and $deg(g(x))=n+1$.  If $f(x)b_{0}=0$, then we may let $g^{\star}(x)=(g(x)-b_{0})/x$ and by our inductive hypothesis see that \emph{R} is indeed $(m,n+1)$-Right McCoy, hence we may assume that $f(x)b_{0}\neq 0$.

Similarly as before, we observe that $a_{0}b_{0}=0$.  Consider an arbitrary index $i$ such that $a_{i}b_{0}\neq 0$.  By setting $f^{\star}(x)=a_{0}f(x)$, we obtain the modified form $a_{0}a_{i}b_{0}$.  Since Left Duo implies semi-commutativity, and since $a_{0}b_{0}=0$, clearly $a_{0}a_{i}b_{0}=0$.  Hence for all indices $i$ such that $a_{i}b_{0}\neq 0$, left multiplication of $f(x)$ by $a_{0}$ allows for $b_{0}$ to annihilate the modified polynomial $f^{\star}(x)=a_{0}f(x)$, hence we may set $g^{\star}(x)=(g(x)-b_{0})/x$ and use our inductive hypothesis to find a non-zero ring element which serves as a right annihilator of $f(x)$.  Therefore we conclude that if \emph{R} is Left Duo and $(m,n)$-Right McCoy, it is necessarily $(m,n+1)$-Right McCoy.  We finally conclude that Left Duo implies $(m+1,n+1)$-Right McCoy and thereby prove that Left Duo rings are necessarily Right McCoy by induction on the degree considerations of \emph{R}.
\end{proof}

Alternatively, we offer a more direct route in the inductive step.
\begin{proof}
Since Left Duo implies Right Linearly McCoy, our basis for induction is complete.  Let \emph{R} be Left Duo and $(m,n)$-Right McCoy.  Let $f(x),g(x)\in\emph{R}[x]\setminus\{0\}$ with $deg(f(x))=m+1$, $deg(g(x))=n+1$, and $f(x)g(x)=0$.  Then we have that $a_{0}b_{0}=0$.  If $a_{0}g(x)=0$ and $f(x)b_{0}=0$ then we are finished by taking $f^{\star}(x)=(f(x)-a_{0})/x$ and $g^{\star}(x)=(g(x)-b_{0})/x$ and using the inductive hypothesis.  Hence we may assume that at least one of $a_{0}g(x)\neq 0$ or $f(x)b_{0}\neq 0$.  In fact, we  simply assume that both $a_{0}g(x)$ and $f(x)b_{0}$ are not equal to zero.

If $a_{0}g(x)\neq 0$ then for any index $j$ such that $a_{0}b_{j}\neq 0$, we set $g^{\star}(x)=g(x)b_{0}$ and notice that, by semi-commutativity, $a_{0}b_{j}b_{0}=0$, so clearly $a_{0}g^{\star}(x)=0$.  Likewise, we can set $f^{\star}(x)=a_{0}f(x)$ and see that $f^{\star}(x)b_{0}=0$.  Moreover, we retain that $f^{\star}(x)g^{\star}(x)=0$, hence it suffices to prove that neither $f^{\star}(x)$ nor $g^{\star}(x)$ are the zero polynomial.

Assume that $g^{\star}(x)=0$.  This implies that $b_{0}^{2}=0$.  Moreover, since $g^{\star}(x)=g(x)b_{0}=0$, for all non-zero coefficients of $g(x)$ we have that $b_{j}b_{0}=0$, whence all non-zero coefficients of $g(x)$ must be precisely $b_{0}$.  Thus $f(x)g(x)=0$ immediately yields $b_{0}$ as a non-zero ring element which is a right annihilator of $f(x)$, thereby rendering the construction of $g^{\star}(x)$ from $g(x)$ unnecessary.

Next assume that $f^{\star}(x)=0$.  This implies that $a_{0}^{2}=0$ and similarly as with the $g^{\star}(x)=0$ case, we obtain that all non-zero coefficients of $f(x)$ must be precisely $a_{0}$.  But if this is the case, then we have that $a_{0}g(x)=0$, again rendering our construction of $f^{\star}(x)$ from $f(x)$ unnecessary.

Therefore we conclude that if \emph{R} is Left Duo, \emph{R} is necessarily Right McCoy.
\end{proof}

We now state the main result of this section.

\begin{thm}\label{thm2}
One-sided Duo rings are necessarily McCoy.
\end{thm}
\begin{proof}
This is a direct consequence of Theorem \ref{thm1}.
\end{proof}

\section{Ring element annihilators}\label{sec3}
In this section we continue to work in the context of a Left Duo base ring, though the results hold for one-sided Duo base rings in general.  From Lemma \ref{lem1} we know that two ring elements $l\in\emph{R}$ which act as a left annihilator of $g(x)$ are precisely $a_{0}^{n+1}$ and $a_{m}^{n+1}$.  We now seek to describe some ring elements $r\in\emph{R}$ which act as a right annihilator of the polynomial $f(x)$.

\begin{thm}\label{thm3}
Let \emph{R} be a one-sided Duo ring with $f(x),g(x)\in\emph{R}[x]\setminus\{0\}$ where $f(x)=\sum\limits^{m}_{i=0}a_{i}x^{i}$ and $g(x)=\sum\limits^{n}_{j=0}b_{j}x^{j}$.  If $f(x)g(x)=0$, then $a_{0}^{n+1}$ and $a_{m}^{n+1}$ are left annihilators of $g(x)$ and $b_{0}^{m+1}$ and $b_{n}^{m+1}$ are right annihilators of $f(x)$.
\end{thm}
\begin{proof}
That $a_{0}^{n+1}$ and $a_{m}^{n+1}$ are left annihilators of $g(x)$ was proven in \cite{Camillo1} and \cite{Nielsen} and was restated above as Lemma \ref{lem1}, thus it suffices to prove that $b_{0}^{m+1}$ and $b_{n}^{m+1}$ are right annihilators of $f(x)$.

The proof is by induction and similar to that of Lemma \ref{lem1}.  Clearly $a_{0}b_{0}=0$ since $f(x)g(x)=0$, so our basis for induction is established.  Now assume that for $l<k$ we have $a_{l}b_{0}^{l+1}=0$.  Consider the coefficient of the $x^{k}$ term.  Since $f(x)g(x)=0$, we have that $\sum\limits^{k}_{i=0}a_{k-i}b_{i}=0$.  Right multiplication by $b_{0}^{k}$ then gives $\sum\limits^{k}_{i=0}a_{k-i}b_{i}b_{0}^{k}=0$ which is equivalent to $a_{k}b_{0}^{k+1}+\sum\limits^{k}_{i=1}a_{k-i}b_{i}b_{0}^{k}=0$.  From here we use the Left Duo property and note that each term $a_{k-i}b_{i}=b_{i}^{\star}a_{k-i}$ where $b_{i}^{\star}\in I_{g(x)}$.  Now, from our inductive hypothesis, we can reduce this to $a_{k}b_{0}^{k+1}=0$, whence $b_{0}^{k+1}$ is a right annihilator of $a_{k}$.  Therefore we conclude that $b_{0}^{m+1}$ is a right annihilator of $f(x)$ whenever $f(x)g(x)=0$ in the polynomial ring over a Left Duo base ring.

That $b_{n}^{m+1}$ is also a right annihilator of $f(x)$ follows by the same inversion process used in Lemma \ref{lem1}.
\end{proof}

What this result implies is that, when our base ring is one-sided Duo, the pitfalls which prevent semi-commutative rings from being McCoy are avoided.  That is, we can find ring element annihilators of the polynomials $f(x)$ and $g(x)$ directly from the coefficients of the polynomials themselves.  Furthermore, the corresponding proof can readily be verified in the context of a Right Duo base ring.

\section{Summary of results and their place in the literature}\label{sec4}
We conclude this paper by analyzing the place of these findings in the existing literature.  We proved that one-sided Duo rings are necessarily McCoy, a significant extension of our knowledge of the relationships between various ring theoretic structures.  But we were also able to describe some of the ring element annihilators of polynomials satisfying the McCoy property in the polynomial ring over a one-sided Duo base ring.  To illustrate how the first result fits into the existing literature, we now present an extended version of the diagram from \cite{Camillo1} which includes our results.  It is worth mentioning that this diagram can serve as an outstanding reference for future studies which offer results on extending this diagram or on the involved ring-theoretic structures.

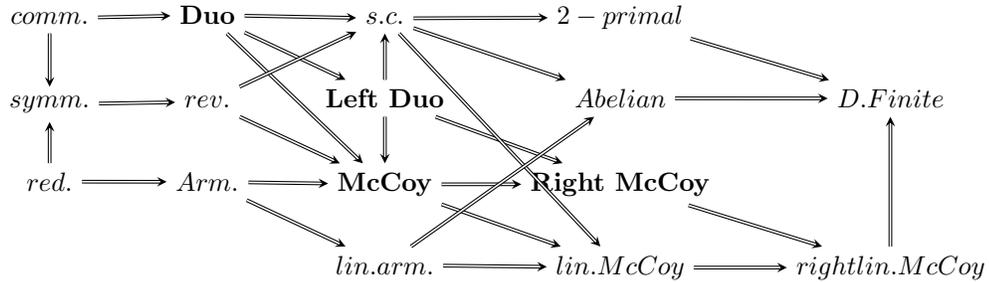
\begin{center}
\begin{tikzpicture}
\matrix (m) [matrix of math nodes, row sep=1.6em, column sep=2.6em, minimum width=2em]
{
comm. & \textbf{Duo} & s.c. & 2-primal\\
symm. & rev. & \textbf{Left Duo} & Abelian & D. Finite\\
red. & Arm. & \textbf{McCoy} & \textbf{Right McCoy}\\
 & & lin. arm. & lin. McCoy & right lin. McCoy\\};
\path[-stealth]
(m-1-1) edge [double] (m-1-2)
(m-1-1) edge [double] (m-2-1)
(m-1-2) edge [double] (m-2-3)
(m-1-2) edge [double] (m-1-3)
(m-1-2) edge [double] (m-3-3)
(m-2-3) edge [double] (m-1-3)
(m-2-3) edge [double] (m-3-3)
(m-2-3) edge [double] (m-3-4)
(m-1-4) edge [double] (m-2-5)
(m-2-1) edge [double] (m-2-2)
(m-2-2) edge [double] (m-1-3)
(m-2-2) edge [double] (m-3-3)
(m-1-3) edge [double] (m-2-4)
(m-1-3) edge [double] (m-4-4)
(m-1-3) edge [double] (m-1-4)
(m-2-4) edge [double] (m-2-5)
(m-3-1) edge [double] (m-2-1)
(m-3-1) edge [double] (m-3-2)
(m-3-2) edge [double] (m-3-3)
(m-3-2) edge [double] (m-4-3)
(m-3-3) edge [double] (m-3-4)
(m-3-3) edge [double] (m-4-4)
(m-3-4) edge [double] (m-4-5)
(m-4-3) edge [double] (m-2-4)
(m-4-3) edge [double] (m-4-4)
(m-4-4) edge [double] (m-4-5)
(m-4-5) edge [double] (m-2-5);
\end{tikzpicture}
\end{center}

Note that, as mentioned in \cite{Camillo1}, if rings are explicitly without unity, then some of these conditions no longer hold in both diagrams.

\section*{Acknowledgments}
The author would like to specifically thank Michal Ziembowski for very thoughtful criticisms which proved crucial to the development of this work.

\bibliographystyle{plain}
\bibliography{duomccoybib}
\end{document}